\documentclass{amsart}

\usepackage{tikz-cd}
\usepackage{hyperref}
\usepackage{amssymb}
\usepackage{comment}
\usepackage{adjustbox}

\usepackage[style=alphabetic]{biblatex}
\newtheorem{thm}{Theorem}[section]
\newtheorem{prop}[thm]{Proposition}
\newtheorem*{prop*}{Proposition}
\newtheorem{cor}[thm]{Corollary}
\newtheorem{lem}[thm]{Lemma}

\theoremstyle{definition}
\newtheorem{defi}[thm]{Definition}
\newtheorem{ex}[thm]{Example}

\newtheorem{rem}[thm]{Remark}

\newtheorem*{nota}{Notation}
\newtheorem*{ack}{Acknowledgements}

\renewcommand{\H}{\text{\normalfont H}}
\renewcommand{\P}{\mathcal P}
\renewcommand{\S}{\mathcal S}
\newcommand{\m}{\mathfrak m}
\newcommand{\p}{\mathfrak p}
\newcommand{\q}{\mathfrak q}

\newcommand{\quotient}[2]{{^{\displaystyle #1}}\Big/{_{\displaystyle #2}}}

\newcommand{\Spec}{\text{\normalfont Spec}\,}
\newcommand{\Ann}{\text{\normalfont Ann}}

\title{A curious characterisation of Dedekind domains}
\author{Robert Szafarczyk}
\date{}

\begin{document}

\begin{abstract}
We characterise Dedekind rings among not necessarily Noetherian domains by a property of their module homomorphisms. Our proof relies on a homological algebra argument.
\end{abstract}

\maketitle

\tableofcontents

\section{Introduction}

In commutative algebra it is often the case that properties of rings imply properties of modules. One of the most important results in the field, the \emph{structure theorem for finitely generated modules over principal ideal domains}, is an example of this. It is furthermore very useful when the converse holds as well. This is not the case for PIDs, however, because the structure theorem is true for a bigger class of rings, the study of which was first undertaken by Kaplansky \cite{Kaplansky} and, after contributions by many authors, culminated with a full description due to Brandal and Wiegand \cite{Brandal}. Other results characterising various rings in terms of their modules appeared in the work of Osofsky (e.g. \cite{Osofsky}), Van Huynh (e.g. \cite{VanHuynh}), and Couchot (e.g. \cite{Couchot1}, \cite{Couchot3}).

In this paper, we give a module-theoretic characterisation of Dedekind domains.

\begin{thm}
Let \(R\) be an integral domain. We say that an \(R\)-module homomorphism \(f:M\to N\) is seemingly divisible by \(r\in R\) if for every \(m\in M\) we have \(f(m)=rn\) for some \(n\in N\) and \(f(m)=0\) whenever \(rm=0\). The following statements are then equivalent:
\begin{enumerate}
\item For all \(r\in R\) and \(f:M\to N\), if \(f\) is seemingly divisible by \(r\), then \(f=r\cdot g\) in \(\text{\normalfont Hom}_R(M,N)\).
\item \(R\) is a Dedekind domain.
\end{enumerate}
\end{thm}

A surprising aspect of our result is that condition (1) forces \(R\) to be Noetherian. This is no longer the case if one removes the domain assumption. A full characterisation of rings satisfying (1) is given in Theorem \ref{thm_main} and examples in Section~\ref{sec_ex}. That \(R\) is Noetherian is neither implied if one restricts (1) to finitely generated modules only. However, this weaker condition would still characterise Dedekind rings among Noetherian domains (see Remark \ref{rem_fin}).

At its core, our proof of (2)\(\implies\)(1) relies on a homological trick, which we present in a simplified case in the following proposition.

\begin{prop} \label{prop_thomas}
A morphism of abelian groups \(f:M\to N\) is divisible by a prime number \(p\in \mathbb Z\) if and only if it is seemingly divisible by \(p\).
\end{prop}

\begin{proof}
Dividing \(f\) by \(p\) is equivalent to lifting it along \(N\xrightarrow{p} N\). Such a lifting problem can be solved in the derived category \(\mathcal D(\mathbb Z)\) where it is equivalent to showing that the composite \(M\to N\to N\otimes_{\mathbb Z}^L \mathbb F_p\) is zero. This is because \(N\otimes_{\mathbb Z}^L \mathbb F_p\) is isomorphic to the cone of \(N\xrightarrow{p} N\). By adjunction, it then suffices to argue that \(f\otimes_{\mathbb Z}^L\mathbb F_p:M\otimes_{\mathbb Z}^L\mathbb F_p\to N\otimes_{\mathbb Z}^L\mathbb F_p\) is zero in \(\mathcal D(\mathbb F_p)\). Since, by assumption, this map is zero on homology and \(\mathbb F_p\) is a field, it must indeed be zero.
\end{proof}

In special cases, for example when \(M\) is finitely generated, the above statement can readily be proven by elementary means. We are not aware, however, of a full argument not equivalent to the one above that circumvents the use of derived categories.

The proof we gave can be interpreted as obstruction theory. More precisely, to the seemingly divisible map \(f\) we attach an obstruction class in \(\text{Ext}_{\mathbb F_p}^1(M/p,N[p])\) that vanishes if and only if \(f\) is divisible by \(p\). This is not that surprising, because we know a priori that the set of divisors of \(f\) is a torsor over \(\text{Hom}_{\mathbb F_p}(M/p,N[p])\).

\begin{nota}
Let \(x\in R\). We denote by \(\Ann(x)\) the kernel of \(R\xrightarrow{x}R\). We write \(V(x)\) for the closed subset of the Zariski spectrum \(\Spec R\) of those prime ideals containing \(x\) and write \(D(x)\) for its complement. If \(M\) is an \(R\)-module, then we denote by \(M[x]\) its \(x\)-torsion submodule.
\end{nota}

\begin{ack}
The author is grateful to Thomas Nikolaus for showing~him the proof of Proposition \ref{prop_thomas}. Further rumination on it lead the author to this project. The author would also like to thank Robert Burklund for several helpful discussions and Semen Slobodianiuk for useful comments on the article's first draft. The author was supported by Danmarks Grundforskningsfond CPH-GEOTOP-DNRF151.
\end{ack}

\section{The characterisation}

\begin{defi}
Let \(R\) be a ring, \(r\) an element of \(R\), and \(f:M\to N\) a homomorphism of \(R\)-modules. We say that \(f\) is \emph{seemingly divisible by} \(r\) if \(f\left(M[r]\right)=0\) and \(f(M)\subset rN\). We say that \(f\) is divisible by \(r\) if it is divisible by \(r\) in the \(R\)-module \(\text{Hom}_R(M,N)\).

We denote by \(\S^R_r\) the collection of all \(R\)-module morphisms that are seemingly divisible by \(r\) and by \(\P^R_r\) the collection of all morphisms divisible by \(r\). We have \(\P^R_r\subset \S^R_r\). If we write \(\S^R=\P^R\), we mean that \(\S^R_r=\P^R_r\) for every \(r\in R\).
\end{defi}

The goal of this section is to characterise rings \(R\) for which \(\S^R=\P^R\). We begin with a few observations.

\begin{rem} \label{rem_P}
Let \(R\) be a ring for which \(\S^R=\P^R\). Then also \(\S^S=\P^S\) for any quotient or localization \(S\) of \(R\). Indeed, this follows from the fact that, up to units, all elements of \(S\) come from \(R\) and that the forgetful functor from \(S\)-modules to \(R\)-modules is fully faithful.
\end{rem}

\begin{rem} \label{rem_prod}
Let \((r,s)\in R\times S\). Then \(\S^{R\times S}_{(r,s)}=\P^{R\times S}_{(r,s)}\) if and only if both \(\S^R_r=\P^R_r\) and \(\S^S_s=\P^S_s\). Observe as well that \(\S^R_e=\P^R_e\) is always satisfied whenever \(e\) is a unit or \(e=0\), so also when \(e\) is an idempotent.
\end{rem}

The following is our main theorem. We give an algebraic characterisation (\ref{pt_2}) as well as a geometric one (\ref{pt_3}). We will later specialize our result to the case of domains in Corollary \ref{cor_dom}.

Recall that a principal ideal ring is a ring in which all ideals are principal. By \cite[Theorem 1]{pid}, all principal ideal rings are finite products of quotients of PIDs.

\begin{thm} \label{thm_main}
For a ring \(R\) the following statements are equivalent:
\begin{enumerate}
\item \(\S^R=\P^R\).
\item Every localization of \(R\) at a prime ideal is a principal ideal ring and for every element \(x\in R\) the ring \(R\) can be written as a product \(R_0\times R_1\times R_2\), where
\begin{itemize}
\item[\(\circ\)] \(x=(0,x_1,x_2)\in R_0\times R_1 \times R_2\).
\item[\(\circ\)] \(x_1\in R_1\) is a non-zero divisor with discrete \(V(x_1)\).
\item[\(\circ\)] \(R_2\) is a principal ideal ring.
\end{itemize} \label{pt_2}
\item Every localization of \(R\) at a prime ideal is a principal ideal ring, all non-reduced points of \(\Spec R\) are isolated, and whenever \(\p\in\Spec R\) is an accumulation point of a subset \(S\subset \Spec R \backslash\{\p\}\) and \(\q\subset \p\), then \(\q\) is an accumulation point of \(S\) as well. \label{pt_3}
\end{enumerate}
\end{thm}

We break up the proof into a series of lemmas. We will first show that (1)\(\implies\)(3), then (3)\(\implies\)(2), and lastly (2)\(\implies\)(1).

The main step in proving (1)\(\implies\)(3) is doing it in the local case. We will do so in Lemma \ref{lem_loc}.

\begin{nota}
Let \(R\) be a ring. By \(R\{y_1,y_2...\}\) we mean the free \(R\)-module \(\bigoplus_{\mathbb N} R\) with generators named \(y_1,y_2,...\) .

Given a sequence of elements \(x_1,x_2,...\in R\) and an \(R\)-module \(M\), by \(M[\frac1{x_1x_2...}]\) we denote the colimit of the diagram
\[M\xrightarrow{x_1}M\xrightarrow{x_2}M\xrightarrow{x_3}...\ .\]
That is, \(M[\frac1{x_1x_2...}]=\{\frac{m}{x_1...x_n} : m\in M, n\geq 0\}/_\sim\) where \(\frac{m}{x_1...x_n}\sim \frac{m'}{x_1...x_k}\) if and only if there exists \(t\geq \max(k,n)\) such that \(x_{n+1}...x_t\cdot m = x_{k+1}...x_t\cdot m'\) in \(M\).
\end{nota}

\begin{lem} \label{lem_tech}
Let \((R,\m)\) be a local ring and let \(x\in \m\). Suppose there is a sequence of elements \(x_1,x_2,...\in \m\) such that the product \(x_1...x_n\) is not contained in \(\left(x,\Ann(x)\right)\) for any \(n\). Then \(\S^R_x\neq\P^R_x\).
\end{lem}

\begin{proof}
 We can consider the following morphism of \(R\)-modules
 
\adjustbox{scale=0.9,center}{\begin{tikzcd} f:R\left[{\textstyle \frac1{x_1x_2...}}\right] \ar[r] & \quotient{\left(\quotient{R}{\Ann(x)}\left[\textstyle \frac1{x_1x_2...}\right] \oplus R\{y_1,y_2,...\}\right)}{\left\{xy_i-\textstyle \frac1{x_1...x_i}\right\}_{i>0}} \end{tikzcd}}
given by the canonical projection onto the first factor of the sum. By construction, it is seemingly divisible by \(x\). We will show that \(f\) is not divisible by \(x\).

Suppose to the contrary that there is \(g\) such that \(f=x\cdot g\). We can write
\[g(1)= \frac{r}{x_1...x_n} + r_1y_1 + r_2y_2 + ... + r_ny_n\]
for some \(n\) and \(r\in R/\Ann(x)\), \(r_k\in R\). Multiplying both sides by \(x\) yields
\[1 = \frac{rx}{x_1...x_n} + \frac{r_1}{x_1} + \frac{r_2}{x_1x_2} + ... + \frac{r_n}{x_1...x_n}\]
in \(R/\Ann(x)[\frac{1}{x_1x_2...}]\). This gives (after potentially increasing \(n\))
\[x_1...x_n = rx + r_1x_2...x_n + r_2x_3...x_n + ... + r_n \label{eq_1} \tag{\(*\)}\]
in \(R/\Ann(x)\).

For each \(j>0\), we have the projection

\adjustbox{scale=0.9,center}{\begin{tikzcd} \quotient{\left(\quotient{R}{\Ann(x)}\left[\textstyle \frac1{x_1x_2...}\right] \oplus R\{y_1,y_2,...\}\right)}{\left\{xy_i-\textstyle \frac1{x_1...x_i}\right\}_{i>0}} \ar[r] & \quotient{R\{y_j\}}{x}.\end{tikzcd}}
Let \(g_j\) denote its pre-composition with \(g\). We then have
\[r_j=g_j(1)=x_1...x_{j+1}\cdot g_j\left(\frac{1}{x_1...x_{j+1}}\right)\]
in \(R/x\), so we get \(r_j=s_jx_1...x_j\) in \(R/x\) with \(s_j=x_{j+1}\cdot g(\frac{1}{x_1...x_{j+1}})\in \m\). This, combined with (\ref{eq_1}), gives
\[(1-s_1-...-s_n)\cdot x_1...x_n\in (x,\Ann(x)).\]
But \((1-s_1-...-s_n)\) is invertible, so \(x_1...x_n\in (x,\Ann(x))\) contradicting our assumptions.
\end{proof}

\begin{lem} \label{lem_loc}
Let \((R,\m)\) be a local ring for which \(\S^R=\P^R\). Then \(R\) is a principal ideal ring. In particular, \(\bigcap\m^n=0\) in \(R\).
\end{lem}

\begin{proof}
Assume \(R\) is not zero nor a field. We will prove the statement in three steps:

1. We show that \(\m\ne\m^2\).

2. We show that \(\m\) is principal.

3. We show that \(\bigcap_{n>0}\m^n=0\).

\noindent Note that, it follows from the last two that \(R\) is principal. Indeed, they imply that every non-zero element of \(R\) is of the form \(ut^k\) for \(u\) a unit and \(t\) a generator of \(\m\).

\emph{Step 1.} Suppose to the contrary that \(\m=\m^2\). We first show that there is a non-zero \(x\in\m\) for which \(\m\ne (x,\Ann(x))\). Indeed, if \(\m=\m^2\) and \(\m=(x,\Ann(x))\), then \(x\in (x^2,\Ann(x))\), so \(x^2\in (x^3)\). This implies \(x^2=0\) by locality, so \(x\in \Ann(x)\). But then \(\m=\Ann(x)\). If this was true for all \(x\in\m\), we would have \(\m=\m^2 = 0\), contradiction. Let us therefore choose a non-zero \(x\in \m\) for which \(\m\ne(x,\Ann(x))\).

We now claim that there is a sequence \(x_1,x_2...\in\m\) such that for any \(n\) the product \(x_1x_2...x_n\) is not contained in \((x,\Ann(x))\). Indeed, as \(\m\ne (x,\Ann(x))\), there is \(x_1'\in \m\backslash (x,\Ann(x))\). Using \(\m=\m^2\), we write \(x_1'\) as a sum of products of elements of \(\m\). At least one of the terms of the sum will still not be contained in \((x,\Ann(x))\), say \(x_1x_2'\). Continuing this process for \(x_2'\) and so on we get our sequence. This contradicts Lemma \ref{lem_tech}, so \(\m\ne \m^2\).

\emph{Step 2.} Choose any \(t\in \m\backslash \m^2\). We will show that \(\m=(t)\). For any \(y\in \m\backslash(t)\) consider the \(R\)-module map
\[\begin{tikzcd} f: R/y \ar[r] & R/\m^2 \end{tikzcd}\]
given by multiplication by \(t\). It is clear that everything in the image of \(f\) is divisible by \(t\). If \(z\in R\) belongs to the \(t\)-torsion of \(R/y\), then \(zt=ay\) for some \(a\in R\). As \(y\not\in (t)\), the element \(a\) cannot be invertible, so \(a\in\m\) implying \(zt\in\m^2\), so that \(f\) vanishes on \(t\)-torsion as well. Hence, \(f\) is seemingly divisible by \(t\), so \(f=t\cdot g\) by assumption. We have \(t\cdot g(1) = f(1) = t\) in \(R/\m^2\). As \(t\not\in \m^2\), this means that \(g(1)\) is invertible, so \(y\in \m^2\), because \(y\cdot g(1) = g(y)=0\) in \(R/\m^2\). Hence, we get that \(\m\backslash (t)\subset \m^2\). Now, if \(\m\ne(t)\), there is a \(b\in\m\backslash(t)\). But also \(t-b\in \m\backslash(t)\) yielding \(t=b+(t-b)\in \m^2\), a contradiction.

\emph{Step 3.} Let \(I=\bigcap_{n>0} \m^n\). Recall that \(\m=(t)\). We will first prove that all elements of \(I\) are \(t\)-power torsion, that is vanish after multiplication with some power of \(t\). Take any \(w\in I\). By considering the constant sequence \(t,t,t...\) in Lemma \ref{lem_tech}, we deduce that there is a \(k\) for which \(t^k\in (w,\Ann(w))\). Hence, we get that \(t^kw=aw^2\) for some \(a\in R\). This yields \(t^kw\, (1-av)=0\), where \(v\in \m\) is such that \(w=t^kv\). But \(1-av\) is invertible, so \(t^kw=0\) as desired.

We will now prove that \(I=0\). This is the case if \(t\) is nilpotent, so assume it is not. Consider the \(R\)-module
\[M=\left\{(r_0,r_1,...) : r_i\in R,\ \forall_k\exists_N\forall_{n>N}\, t^k|r_n\right\}\]
and the \(R/I\)-module
\[N=\left\{(r_0,r_1,...) : r_i\in R/I,\ \forall_k\exists_N\forall_{n>N}\, t^k|r_n\right\}.\]
There is a surjection \(M\twoheadrightarrow N\) that sends each \(r_i\) to its class modulo \(I\). Now, assume \(I\ne 0\), so that there exists a non-zero \(s\in I\) that is \(t\)-torsion. Multiplication by \(s\) on \(M\) has image in the submodule of finite sequences, hence induces an \(R\)-module map
\[\begin{tikzcd} f: M\ar[r] & \bigoplus_{\mathbb N} R. \end{tikzcd}\]
It is seemingly divisible by \(s\), so \(f=s\cdot g\) by assumption. Let \(h\) denote the post-composition of \(g\) with the projection \(\bigoplus_{\mathbb N} R \to \bigoplus_{\mathbb N} R/I\). Since \(\bigoplus_{\mathbb N}R/I\) has no non-zero infinitely \(t\)-divisible elements and the kernel of \(M\twoheadrightarrow N\) consists only of such, we get that \(h\) factors through \(N\). Observe that, because \(f=s\cdot g\) and \(\Ann(s)=(t)\), whenever an element of \(N\) is not \(t\)-divisible, then its image under \(h\) is non-zero. As every element of \(N\) can be maximally divided by a power of \(t\) and \(\bigoplus_{\mathbb N}R/I\) is \(t\)-torsion free, this implies that \(h:N\to \bigoplus_{\mathbb N}R/I\) is injective. But, since \(R/I\) is a PID, this would mean that \(N\) is countably free as an \(R/I\)-module, which it is not. Hence, we must have \(I=0\).
\end{proof}

After addressing the local case, the next step is to control the geometry of \(\Spec R\).

\begin{lem} \label{lem_conv}
Let \(R\) be a ring for which \(\S^R=\P^R\). Let \(x\in R\) and let \(\p\) be a point of \(\Spec R\) whose every open neighbourhood contains some \(\q\) with \(\p\not\subset \q\) and \(x\in \q\). Then \(x\) is zero in \(R_\p\).
\end{lem}

\begin{proof}
Combining Remark \ref{rem_P} and Lemma \ref{lem_loc}, it suffices to show that \(x\in \bigcap_{n\geq 0} \p^n\). Assume that \(x\in\p^k\) for some \(k\). Let \(I\) be the intersection of all \(\q\) with \(\p\not\subset \q\) and \(x\in \q\). By assumption, we have \(\p\in V(I)\), so \(I\subset \p\). Consider the \(R\)-linear map
\[\begin{tikzcd}
f:R/\p^{k+1}\ar[r] & R/I\cap \p^{k+1}
\end{tikzcd}\]
given by multiplication by \(x\). If \(y\in R\) is in the \(x\)-torsion of the source, then \(xy\in\p^{k+1}\), so also \(xy\in I\cap\p^{k+1}\) as \(x\in I\). That is, \(f\) vanishes on \(x\)-torsion. Thus, it is seemingly divisible by \(x\), so \(f=x\cdot g\) by assumption. As \(g(\p^{k+1})=0\), we have \(g(1)\cdot \p^{k+1}\subset I\cap \p^{k+1}\). Hence, \(g(1)\in I\subset \p\), and so \(x\cdot g(1)\in I\cap \p^{k+1}\). Thus, as \(x = f(1) = x\cdot g(1) = 0\) in \(R/I\cap\p^{k+1}\), we get \(x\in \p^{k+1}\), so indeed \(x\in \bigcap_{n\geq 0} \p^n\).
\end{proof}

\begin{lem} \label{lem_dom}
Let \(R\) be a ring for which \(\S^R=\P^R\).  Let \(\p\) be a non-isolated point of \(\Spec R\). Then \(R_\p\) is a domain.
\end{lem}

\begin{proof}
The assertion follows from Lemma \ref{lem_loc} whenever \(\p\) is not maximal. Let us therefore assume that \(\p\) is maximal. Suppose there are \(x,y\in R\) such that \(xy=0\) in \(R_\p\). The same is then true in some open neighbourhood of \(\p\), but this means that either \(x\) or \(y\) satisfies the conditions of Lemma \ref{lem_conv}, so \(x=0\) or \(y=0\) in \(R_\p\). 
\end{proof}

We can now prove the first two implications from Theorem \ref{thm_main}.

\begin{proof}[Proof of (1)\(\implies\)(3) (Theorem \ref{thm_main})]
That every localization of \(R\) at a prime is principal follows from Lemma \ref{lem_loc} and that all non-reduced points of \(\Spec R\) are isolated by Lemma \ref{lem_dom}. Suppose now we are given \(\p\), \(\q\), and \(S\subset \Spec R\backslash\{\p\}\) such that \(\p\) is an accumulation point of \(S\) and \(\q\subset\p\). We want to show that \(\q\) is an accumulation point of \(S\). Assume to the contrary, that it is not. Then there exists \(x\in R\) such that \(\q\in D(x)\) and \(S\subset V(x)\). As \(\p\) is not contained in any ideal from \(S\), because \(R\) is one-dimensional by Lemma \ref{lem_loc}, it follows from Lemma \ref{lem_conv} that \(x=0\) in \(R_\p\). But then \(x\in \q\), contradiction.
\end{proof}

\begin{proof}[Proof of (3)\(\implies\)(2) (Theorem \ref{thm_main})]
Let \(x\in R\). We want to prove the decomposition from condition (2). Consider the largest open subset \(U\subset \Spec R\) restricted to which \(x=0\). It can equivalently be described as those \(\p\in\Spec R\) such that \(x=0\) in \(R_\p\). Note that whenever \(\q\) is an accumulation point of some subset \(S\subset V(x)\) not containing \(\q\), then \(\q\in U\). Indeed, by assumption, every prime ideal contained in \(\q\) is also an accumulation point of \(S\). In particular, \(x\) is contained in every prime ideal of \(R_\q\). But \(R_\q\) is reduced, because non-reduced points are isolated. Hence, \(x=0\) in \(R_\q\). This observation implies that \(U\) is clopen and that \(V(x)\backslash U\) is finite and discrete. Therefore, we can write \(R=R_0\times R'\) with \(x=(0,x')\) and such that \(V(x')\) is discrete (so necessarily finite, because quasi-compact).

We now work with \(R'\). Consider the intersection of \(V(x')\) with the non-reduced locus. It is finite and consists of isolated points only, so is clopen. We can therefore write \(R'=R_1\times R_2\) with \(x=(x_1,x_2)\) such that \(x_1\) is non-nilpotent in \((R_1)_\p\) for any \(\p\) and where \(R_2\) is a finite product of principal rings, hence is principal itself. A non-nilpotent element of a local principal ring is not zero nor a zero divisor. This implies that \(x_1\) is a non-zero divisor in \(R_1\), which proves our decomposition.
\end{proof}

For the last implication we will need some technical lemmas.

\begin{lem} \label{lem_dec}
Let \(R\) be a ring and let \(\p\subset R\) be a finitely generated maximal ideal for which \(R_\p\) is a principal ideal ring. Then, for any \(n\geq 1\), every \(R\)-module \(M\) admits a decomposition of the form \[M\simeq (R/\p)^{\oplus I_1}\oplus (R/\p^2)^{\oplus I_2}\oplus...\oplus (R/\p^{n-1})^{\oplus I_{n-1}} \oplus M'\] with \(M'\) such that \(M'/\p^n\) is a free \(R/\p^n\)-module.
\end{lem}

\begin{proof}
Note that \(R/\p^m\simeq R_\p/\p^m\), hence, as \(R_\p\) is principal, \(R/\p^m\) is self-injective (this can be seen using Baer's criterion). We will show by induction that \(M\simeq (R/\p)^{\oplus I_1}\oplus...\oplus (R/\p^{m-1})^{\oplus I_{m-1}} \oplus M^{(m)}\) where \(M^{(m)}\) satisfies the following property \((*_m)\): whenever \(x\in M^{(m)}\) is such that \(\p^{m-1} x=0\), then \(x\in\p M^{(m)}\).

Before we do so, let us remark that if \(M\) satisfies \((*_m)\), then \(M/\p^m\) as well. Indeed, first note that if \(M\) satisfies \((*_m)\), then \(M_\p\) as well\footnote{We use finite generation of \(\p\) here.}. Let \(p\) be a generator of \(\p\) in \(R_\p\). It then suffices to see that if \(x\in M_\p\) is such that \(p^{m-1}x=p^my\), then \(x\in p M_\p\). This follows by \((*_m)\) applied to \(x-py\).

The case \(m=1\) is vacuous. If \(M^{(m-1)}\) satisfies \((*_{m-1})\) but not \((*_m)\), then we can find an injective map \(R/\p^{m-1}\hookrightarrow M^{(m-1)}\) such that the composite \(R/\p^{m-1}\to M^{(m-1)} \to M^{(m-1)}/\p\) is non-zero. Consider the composite \(R/\p^{m-1}\to M^{(m-1)} \to M^{(m-1)}/\p^{m-1}\). If it is not injective, then, by property \((*_{m-1})\) for \(M^{(m-1)}/\p^{m-1}\), we learn that the composite \(R/\p^{m-1}\to M^{(m-1)} \to M^{(m-1)}/\p^{m-1}\to M^{(m-1)}/\p\) is zero, contradiction. Hence, it is injective, so, by self-injectivity of \(R/\p^{m-1}\), it splits and we can repeat this process until the remaining module satisfies \((*_m)\).

It remains to show that if \(M\) is an \(R/\p^m\)-module satisfying \((*_m)\), then it is free. Let \(I\) be a basis for the vector space \(M/\p\). Choosing lifts gives a surjective map \(f:\bigoplus_IR/\p^m \to M\). We will show by induction that \(f/\p^k\) is injective for all \(k\geq 1\). This is enough as \(f/\p^m=f\). Case \(k=1\) holds by construction. Suppose now that \(f/\p^{k-1}\) for \(k\leq m\), is injective but there are \(a_i\in R/\p^m\), indexed by a finite subset of \(I\), such that \(\sum_i a_i f(1_i)=p^kx\) for some \(x\in M\). By the inductive assumption, this implies that \(a_i\in \p^{k-1}\), so \(a_i = p^{k-1}b_i\) for some \(b_i\). But then \(\sum_i b_if(1_i) - px\) is \(p^{k-1}\)-torsion, so, by property \((*_m)\), the sum \(\sum_i b_if(1_i)\) is divisible by \(p\), so also all \(b_i\) are. This means that \(a_i \in \p^k\), so \(f/\p^k\) is injective, which finishes the proof.
\end{proof}

In the next lemma, we will denote by \(\Sigma\) a homological shift by 1.

\begin{lem} \label{lem_nul}
Let \(R\) be any ring. Let \(g:A\to B\) be a morphism in the derived category \(\mathcal D(R)\). Assume that \(A\) and \(B\) have homology concentrated in degrees \([0,1]\). If \(f\) induces the zero morphism on homology, then it can be factored through some map \(\H_0(A)\to \Sigma\H_1(B)\).

In particular, if \(\H_0(A)\) is a projective \(R\)-module, then \(g\) is zero.
\end{lem}

\begin{proof}
Using the canonical t-structure, as \(A\to \H_0(B)\) is zero, we can lift \(f\) along \(\Sigma \H_1(B) \to B\). Similarly, as \(\H_1(A)\to \H_1(B)\) is zero, \(A\to \Sigma \H_1(B)\) can be factored through \(\H_0(A)\).
\end{proof}

\begin{lem} \label{lem_der}
Let \(x\in R\) be a non-zero divisor and let \(f:M\to N\) be a map of \(R\)-modules. Then \(f\) is divisible by \(x\) if and only if \(f\otimes_R^L R/x\) is zero in \(\mathcal D(R/x)\).
\end{lem}

\begin{proof}
Since \(x\) is a non-zero divisor, the derived tensor product \(N\otimes_R^L R/x\) can equivalently be described as the cone of \(N\xrightarrow{x}N\) in \(\mathcal D(R)\). Hence, \(f\) is divisible by \(x\) if and only if the composite \(M\xrightarrow{f} N \to N\otimes_R^L R/x\) is zero. This is the same, by adjunction, as \(f\otimes_R^LR/x\) being zero in \(\mathcal D(R/x)\).
\end{proof}

We can now show that (2)\(\implies\)(1) and finish the proof of Theorem \ref{thm_main}.

\begin{proof}[Proof of (2)\(\implies\)(1) (Theorem \ref{thm_main})]
Let \(x\in R\). Note that, in our decomposition, \(R_2\) is a principal ring, so a quotient of a finite product of PIDs. If we replace \(R_2\) with this product, then in each of those PIDs \(x\) is either zero or a non-zero divisor. Hence, using Remark \ref{rem_P}, we can reduce to the case \(R = R_0\times R_1\). We can further assume \(R=R_1\) by Remark \ref{rem_prod}.

Now, let \(f:M\to N\) be a map of \(R\)-modules that is seemingly divisible by \(x\). We want to show it is divisible by \(x\). As explained above, we can assume \(R=R_1\), so that \(x\) is a non-zero divisor with discrete \(V(x)\). In that case, we have \(R/x\simeq \prod_{\p\in V(x)} R_\p/x\). By Lemma \ref{lem_der}, both division of \(f\) and of \(f\otimes_R (\prod_{\p\in V(x)} R_\p)\) by \(x\) is equivalent to \(f\otimes_R^LR/x\) being zero. Hence, we can replace \(R\) with \(\prod_{\p\in V(x)} R_\p\), so also just with \(R_\p\), again by Remark \ref{rem_prod}.

As now \(R\) is a principle ideal ring, we have \((x)=\p^n\) for some \(n\) and we can factor \(M\) like in Lemma \ref{lem_dec}. Since all but the last factor are \(x\)-torsion, and a seemingly divisible map must be zero on those, this reduces us to the case of \(M\) with \(M/x\) projective which follows from Lemma \ref{lem_der} combined with Lemma \ref{lem_nul}.
\end{proof}

If \(R\) is further assumed to be a domain, we can get a more explicit result.

\begin{cor} \label{cor_dom}
The following are equivalent for an integral domain \(R\):
\begin{enumerate}
\item \(\S^R=\P^R\).
\item \(R\) is a Dedekind domain.
\end{enumerate}
\end{cor}

\begin{proof}
If \(R\) is a domain, then condition (2) of Theorem \ref{thm_main} becomes equivalent to \(R\) being locally PID and \(V(x)\) being finite for every non-zero \(x\in R\). This is equivalent, by \cite[Theorem 3]{almDed}, to \(R\) being a Dedekind domain.
\end{proof}

\begin{rem} \label{rem_fin}
One can see from the proof of Lemma \ref{lem_loc} that among Noetherian domains Dedekind ones are already characterised by \(\S^{R}=\P^R\) restricted to maps between finitely generated modules. For non-Noetherian domains, e.g.\@ \(R=k[\sqrt[2^\infty]{x}]\), it can happen that \(\S^R=\P^R\) for finitely presented modules, but not in general.
\end{rem}

\section{Examples} \label{sec_ex}

In this short section, we provide some examples of rings \(R\) that satisfy \(\S^R=\P^R\). Corollary \ref{cor_dom} gives a clear answer in the case of domains, but Theorem \ref{thm_main} is less explicit in general.

Let us first spell out what we know about the geometry of \(\Spec R\). By condition~(3) of Theorem \ref{thm_main}, there are three types of irreducible components of \(\Spec R\): reduced points, non-reduced points and 1-dimensional components.

Our first example shows that there is no restriction on reduced points.

\begin{ex}
Let \(R\) be a von Neumann regular ring, that is a 0-dimensional reduced ring. Then, all its elements, up to units, are idempotents, so \(\S^R=\P^R\) by Remark \ref{rem_prod}. 
\end{ex}

Condition (3) gives, however, restrictions on convergence in relation to non-reduced and 1-dimensional components. We illustrate in the next example that all cases of convergence not excluded by it can indeed appear.

\begin{ex}
Let \(k\) be a field. Consider \(R\), the ring of eventually constant sequences...
\begin{enumerate}
\item valued in \(k(x)\) whose limit is a polynomial.
\item valued in \(k[x]\) whose limit is a constant.
\item valued in \(k[x]/x^n\) whose limit is a constant.
\item valued in \(k(x)[y]\) whose limit is a polynomial in \(x\) and a constant in \(y\).
\item valued in \(k(x)[y]/y^n\) whose limit is a polynomial in \(x\) and a constant in \(y\).
\end{enumerate}
The spectrum of \(R\) looks like...
\begin{enumerate}
\item infinitely many reduced points converging to the generic point of a line.
\item infinitely many lines converging to a reduced point.
\item infinitely many non-reduced points converging to a reduced point.
\item infinitely many lines converging to the generic point of a line.
\item infinitely many non-reduced points converging to the generic point of a line.
\end{enumerate}
In all cases \(R\) satisfies condition (3) of Theorem \ref{thm_main}, so \(\S^R=\P^R\). 
\end{ex}

We give the last example to show that some of the aforementioned convergences can occur in a very chaotic way.

\begin{ex}
Let \(k\) be an infinite field. Let \(\mathcal C\) denote the Cantor set. For every point \(c\in \mathcal C\) we fix a locally constant function \(f_c\) on \(\mathcal C\backslash\{c\}\) valued in \(k\) for which every level set is clopen in \(\mathcal C\). This can readily be done as \(\mathcal C\) is totally disconnected. We extend \(f_c\) to all of \(\mathcal C\) by \(f_c(c)=0\). Observe, that for every convergent sequence \(c_i\to c\) with \(c_i\in \mathcal C\backslash\{c\}\) only finitely many \(c_i\) satisfy \(f_c(c_i)=\lambda\) for any fixed \(\lambda \in k\), so \(f_c\) satisfies no polynomial equation on any open neighbourhood of \(c\).

We then consider the ring \(R\) of functions on \(\mathcal C\) that locally around every point \(d\in\mathcal C\) can be expressed as a polynomial in \(f_d\). The ring \(S\) of locally constant functions on \(\mathcal C\) includes into \(R\). In particular, \(f_c\in R\) as it is locally constant away from \(c\). The fibre of the ring map \(S\to R\) over the maximal ideal of those locally constant functions that vanish at some \(e\in\mathcal C\) is isomorphic to \(k[f_e]\). Hence, the spectrum of \(R\) looks like infinitely many lines parametrized by the Cantor set and whenever there is a non-trivial convergent sequence \(x_i\to x\) in \(\mathcal C\), then the lines over \(x_i\) converge to the generic point of the line over \(x\) (this is because for every \(f\in k[f_{x}]\) there is an open neighbourhood \(U\) of \(x\) such that \(f\) restricted to \(U\backslash\{x\}\) is invertible). Hence, \(R\) satisfies condition (3) of Theorem \ref{thm_main}, so \(\S^R=\P^R\).
\end{ex}

\printbibliography

\end{document}